\documentclass{amsart}
\usepackage{amsmath,amsthm}
\usepackage{amssymb}

\usepackage{pgfplots}
\pgfplotsset{compat=1.15}
\usepackage{mathrsfs}
\usetikzlibrary{arrows}
\usetikzlibrary[patterns]

\usepackage{enumitem}

\newtheorem{theorem}{Theorem}[section]

\theoremstyle{definition}
\newtheorem{definition}[theorem]{Definition}
\newtheorem{example}[theorem]{Example}
\newtheorem{remark}[theorem]{Remark}

\newcommand{\R}{\mathbb{R}}
\newcommand{\Lip}{\operatorname{Lip}}
\newcommand{\Li}{\operatorname{Li}}
\newcommand{\supp}{\operatorname{supp}}
\newcommand{\Measp}{\mathcal{M}_{1}(X)}

\newcommand{\notion}[1]{\emph{#1}}
\newcommand{\cl}[1]{\operatorname{cl}\,{#1}}
\newcommand{\inte}[1]{\operatorname{int}\,{#1}}

\begin{document}



\title{Strongly-Fibred Iterated Function Systems and 
the Barnsley--Vince triangle}

\title{Strongly-Fibred Iterated Function Systems and 
the Barnsley--Vince triangle}
\author[K. Le\'{s}niak]{Krzysztof Le\'{s}niak}
\author[N. Snigireva]{Nina Snigireva}
\author[F. Strobin]{Filip Strobin}



\begin{abstract} We review the theory of semiattractors
associated with non-contractive 
Iterated Function Systems (IFSs) and 
demonstrate its applications on a concrete example. 
In particular, we present criteria for 
the existence of semiattractors due to Lasota and Myjak. 
We also discuss the Kieninger criterion 
which allows us to characterise when 
a semiattractor is strongly-fibred. 
Finally, we consider a specific example of 
a non-contractive IFS introduced by Barnsley and Vince. 
We find an invariant measure for this system 
which allows us to describe its semiattractor. 
The difficulty in analysing this IFS stems 
from the fact that it is neither eventually contractive 
nor contractive on average.
\end{abstract}
\maketitle

\section{Introduction} 

 

Iterated function systems (IFS) provide a general framework 
for generating fractals. Traditionally, an IFS
consists of a finite set of Banach contractions 
acting on a complete metric space. Almost from 
the beginning of the modern theory of IFSs, 
introduced by Hutchinson in 1981 and 
later developed by a number of researches (see \cite{LSS-Manual}
and the references therein for an overview), it was 
allowed for a possibility that at least
some maps comprising an IFS were not contractive.
The most widely accepted approach to study IFSs is to switch 
from the dynamics on sets to the dynamics on
measures and consider the IFSs that 
are contractive on average, see for example, \cite{BarnsleyElton}. In this setting, invariant and asymptotically stable measures play the roles of invariant sets and attractors, respectively. Lasota and Myjak provided a  unified framework to study IFSs which are contractive on average by developing the theory of semiattractors, 
see \cite{LasotaMyjak, LasotaMyjak2000}.
Meanwhile, the concept of topological contractivity 
has started to crystalize in various works, see for example \cite{Ed, Kam}, 
culminating in the Kieninger's classification of invariant sets as point-fibred or strongly-fibred, 
see \cite{Kieninger}. Very recently, point-fibred minimal invariant
sets turned out to be attractors 
of IFSs consisting of weak contractions 
(which are not very far from Banach contractions, 
see for instance \cite{BKNNS, Miculescu}). On the other hand, strongly-fibred minimal invariant sets
that are not point-fibred 
cannot be attractors of (weakly) contractive IFSs. However, such sets can be attractors of IFSs 
for which the chaos game algorithm works
(see \cite{BL, BarrientosEtAl2017, BarrientosEtAl2020}). In order to avoid restricting the whole space on which an IFS acts to an invariant set, one often needs to extend the concept of an attractor. A proper natural extension is the concept of a semiattractor,
as illustrated in this paper by analysing an example of a strongly-fibred triangle
introduced by Barnsley and Vince, see \cite{BV2013}.

\section{Preliminaries}

In this section we will introduce the notion 
of an iterated function system (IFS)
together with its probabilistic extension. 
We will also define a strict attractor, a semiattractor,
and an invariant probability measure of an IFS. 

Throughout this paper $X$ stands for a Polish space. 
However, let us note that many results stated here, 
namely those which do not involve measures and metrics, 
are valid for general topological spaces. 
In that setting one only needs to replace Hausdorff 
convergence with Vietoris convergence; e.g., \cite{BL}.

The Lipschitz constant of $f:X\to X$ is denoted by $\Lip(f)$.
The \notion{lower Kuratowski limit} of a sequence of sets
$S_n\subseteq X$, $n\geq 1$, is the set
\begin{align*}
	\Li S_n = \{y\in X: y=\lim_{n\to\infty} x_n 
	\;\mbox{for some sequence}\; x_n\in S_n\}; 
\end{align*}
see \cite{Beer}.
By $\Measp$ we denote the space of Borel probability measures
$\mu$ on $X$. The support of $\mu$ is denoted by $\supp \mu$.

An \notion{iterated function system} (IFS) $\mathcal{F}$ 
is a system of continuous maps $f_i:X\to X$, $i\in I$, 
where $I$ is a finite indexing set. 
We write $\mathcal{F}=(X; f_i:i\in I)$. 
The associated \notion{Hutchinson operator} 
$F: 2^X\to 2^X$ is given by 
\begin{align*}
	F(S)=\cl{\bigcup_{i\in I} f_i(S)}
\end{align*} 
for $S\subseteq X$, 
where $2^X$ stands for the power set of $X$ 
and $\cl{\!}$ denotes the topological closure of a set. 
To simplify notation we write $F(x)$ instead of $F(\{x\})$.

A \notion{probabilistic IFS} 
$\mathcal{F}p = (X; (f_i,p_i):i\in I)$ 
is the IFS $\mathcal{F}=(X; f_i:i\in I)$ 
together with a vector of probabilities 
$(p_i)_{i\in I}$, $p_i\in [0,1]$, 
$\sum_{i\in I} p_i = 1$.
The associated \notion{Markov operator}
$M:\Measp\to\Measp$ is given by 
\begin{align*}
	M\mu =  \sum_{i\in I} p_i\cdot \mu\circ{f_i}^{-1}
\end{align*}
for ${\mu\in\Measp}$, where $\mu\circ{f_i}^{-1}$ 
is the push-forward measure through $f_i$. 

We are now ready to define 
a strict attractor (\cite{BL, BV2013}) 
and a semiattractor (\cite{LasotaMyjak2000, MyjakSzarek}).

\begin{definition}
Let $\mathcal{F}=(X; f_i:i\in I)$ be an IFS and 
$F$ the associated Hutchinson operator; 
$F^n$ stands below for the $n$-fold composition of $F$.
A nonempty closed set $A_{*}\subseteq X$ is 
\begin{itemize}
	\item[(i)] an \notion{invariant set} 
	if $F(A_{*})=A_{*}$; 
	\item[(ii)]	a \notion{minimal invariant set} 
	if $A_{*}$ is an invariant set and 
	there is no nonempty closed invariant set $C$ that 
	is a proper subset of $ A_{*}$, 
	i.e., $C\subsetneq A_{*}$;
	\item[(iii)] a \notion{strict attractor} if
	there exists an open $U\supseteq A_{*}$ such that
\begin{equation}\label{eq:strictattractor}
	F^n(S) \to A_{*} 
	\;\text{ for all nonempty compact }\;
	S\subseteq U 
\end{equation}
	where the convergence takes place 
	with respect to the Hausdorrf metric
	(see \cite{Beer, Mendivil});
	the largest open $U$ in \eqref{eq:strictattractor}
	is called the \notion{basin} of $A_{*}$
	and denoted by $\mathcal{B}(A_{*})$;  
	\item[(iv)] a \notion{semiattractor}, when 
	$\bigcap_{x\in X} \Li F^n(x) = A_{*}$.
\end{itemize} 
\end{definition}

By an attractor (in a given sense) of 
the probabilistic IFS $\mathcal{F}p$, 
we understand an attractor of 
the underlying IFS $\mathcal{F}$, 
regardless of the choice of probabilities. 
It is worth noting that, in general, 
the semiattractor may be unbounded, 
see \cite[Example 6.2]{LasotaMyjak}.

The relation between various kinds of attractors 
is summarized in the next theorem.

\begin{theorem}[\cite{MyjakSzarek, LasotaMyjak2000,%
		BLR2016}]\label{th:relationsbetweenattractors}
Let $\mathcal{F}=(X; f_i:i\in I)$ be an IFS and 
$A_*\subseteq X$ a nonempty closed set.
\begin{itemize}
	\item[(i)] A strict attractor $A_{*}$ is 
	a compact invariant set which is unique 
	within its basin $\mathcal{B}(A_{*})$.
	\item[(ii)] A semiattractor $A_{*}$ of $\mathcal{F}$ is 
	the smallest invariant set, i.e., $A_{*}\subseteq C$ for 
	every nonempty closed invariant set $C$. Even more,
	$A_{*}\subseteq C$ for every nonempty closed set $C$
	with $F(C)\subseteq C$.
	\item[(iii)] If $A_{*}$ is a strict attractor 
	with basin $\mathcal{B}(A_{*})$, then 
	$F(\mathcal{B}(A_{*}))\subseteq \mathcal{B}(A_{*})$
	and $A_{*}$ is a compact semiattractor
	of the restricted IFS 
	$\mathcal{F}|{\mathcal{B}(A_{*})}= 
	(\mathcal{B}(A_{*}); {f_{i}|}_{\mathcal{B}(A_{*})}:i \in I)$.
	\item[(iv)] If $A_{*}$ is a compact semiattractor 
	of $\mathcal{F}$, then $A_{*}$ is a strict attractor 
	of the restricted IFS 
	$\mathcal{F}|{A_{*}}= (A_{*}; {f_{i}|}_{A_{*}}:i \in I)$.
\end{itemize}
\end{theorem}

Let us now introduce the notion of 
an invariant measure for an IFS (see 
\cite{LasotaMackey, LasotaMyjak, LasotaMyjak2000, Mendivil, Zaharopol}).

\begin{definition}
Let $\mathcal{F}p$ be a probabilistic IFS and 
$M$ the associated Markov operator; 
$M^n$ stands below for the $n$-fold composition of $M$.
We say that a probabilistic measure $\mu_{*}\in\Measp$ is 
\begin{itemize}
	\item[(i)] an \notion{invariant measure}, 
	when $M \mu_{*}=\mu_{*}$,
	\item[(ii)] an \notion{asymptotically stable measure}, 
	when $M^n \mu \to \mu_{*}$
	for every probabilistic measure $\mu\in\Measp$,
	where $\to$ denotes the weak convergence
	of measures (see \cite{Billingsley}).
\end{itemize}
In case (ii) we say that $M$ itself is 
\notion{asymptotically stable}.
\end{definition}

Observe that any asymptotically stable measure 
is a unique invariant measure.
The reason for this is that Markov operators induced 
by probabilistic IFSs are weakly continuous, 
see \cite[Section 3 p.483]{MyjakSzarek}.

\section{Existence of Semiattractors}\label{sec:semiattractors}

In this section we provide two existence criteria 
for semiattractors using the higher iterates of an IFS. 
The first criterion is purely topological one.

\begin{theorem}\label{semiattractor}
Let $\mathcal{F}=(X; f_i: i\in I)$ be an IFS. 
If, for some positive integer $k$, 
$A_{*}$ is a semiattractor of 
$\mathcal{F}^k = (X; f_{\underline{i}}:
\underline{i}\in I^k)$, then 
$A_{*}$ is a semiattractor of $\mathcal{F}$.
\end{theorem}
\begin{proof}
Let us first show that 
\begin{align}\label{eq:AstarIsSubset}
	\bigcap_{x\in X}\Li F^n(x)\supseteq 
	A_{*} \neq\emptyset.
\end{align}
Denote by $G$ the Hutchinson operator induced by 
the IFS $\mathcal{F}^k$. Then $G=F^k$. 
Let $z\in A_{*}=\bigcap_{x\in X}\Li G^n(x)$.
Fix also $x\in X$. For every $0\leq j\leq k-1$ 
there exists a sequence
\begin{align*}
	y_{n}^{(j)}\in G^n(y), 
	\;\text{where}\; 
	y=f_{1}^j(x),
\end{align*}
such that $z=\lim_{n\to\infty} y_{n}^{(j)}$.
(Conventionally $f_{1}^0(x)=x$.) Putting
\begin{align*}
	x_{nk+j} := y_{n}^{(j)}\in F^{nk+j}(x)
	\;\text{for}\; 
	0\leq j\leq k-1, n\geq 1,
\end{align*}
yields $z=\lim_{n\to\infty} x_n$, $x_n\in F^n(x)$.
Therefore $z\in\Li F^n(x)$. Being $x\in X$ arbitrary,
$z\in \bigcap_{x\in X}\Li F^n(x)$.
	
At this stage, by \eqref{eq:AstarIsSubset},
we know that $\mathcal{F}$ admits a semiattractor.
To finish the proof it is enough to 
prove the inclusion that is opposite 
to \eqref{eq:AstarIsSubset}. Indeed, we have 	
\begin{align*}
	\bigcap_{x\in X} \Li F^n(x)
	\subseteq 
	\bigcap_{x\in X} \Li F^{kn}(x)
	\subseteq 
	\bigcap_{x\in X} \Li G^n(x) 
	= A_{*}.
\end{align*} 
\end{proof}

Using the above theorem, we can now improve 
the Lasota--Myjak criterion 
\cite[Theorem 6.1]{LasotaMyjak2000}
for the existence 
of semiattractors and asymptotically stable measures.

\begin{theorem}[Higher iterate Lasota--Myjak criterion]\label{th:criterion3}
Let $\mathcal{F}p =(X; (f_i, p_i):i\in I)$ 
be a probabilistic IFS with positive probabilities $p_i>0$
that consists of Lipschitz maps $f_i$.
Let	$\mathcal{F}p^k =(X; (f_{\underline{i}}, p_{\underline{i}}):
\underline{i}\in I^k)$, $k\geq 1$, be its $k$th iterate,
where $\underline{i}=(i_{1},\dots,i_{k})$,
$f_{\underline{i}} = f_{i_{1}} \circ\dots\circ f_{i_{k}}$,
$p_{\underline{i}} = \prod_{l=1}^{k} p_{i_{l}}$.
Suppose that $\mathcal{F}p^k$ is average contractive, i.e.,
\begin{align*}
	\sum_{\underline{i}\in I^k} p_{\underline{i}} 
	\cdot \Lip(f_{\underline{i}}) <1.
\end{align*}
Then we have
\begin{enumerate}
	\item[(a)] $\mathcal{F}p$ and $\mathcal{F}p^k$ 
	admit semiattractors and these semiattractors coincide;
	\item[(b)] the Markov operator induced by $\mathcal{F}p^k$
	coincides with $M^k$ where $M$ is the Markov operator 
	induced by $\mathcal{F}p$ and both operators 
	are asymptotically stable and share 
	the same unique invariant measure;
	moreover, the invariant measure 
	supports the semiattractors	in (a).	 
\end{enumerate}
\end{theorem}
\begin{proof}
Part (a). 
By the average contractivity, 
there exists at least one $\underline{i}\in I^k$
such that $\Lip(f_{\underline{i}}) <1$.
Then by \cite[Theorem 3.2]{LasotaMyjak2000} 
(which can be also found as \cite[Theorem 6.3]{MyjakSzarek})
we know that $\mathcal{F}p^k$ admits a semiattractor.
(Note that we could have used 
\cite[Theorem 6.1]{LasotaMyjak2000} as well.)
This semiattractor is also the semiattractor of 
$\mathcal{F}p$ according to Theorem \ref{semiattractor}.

Part (b). 
It is readily seen that $M^k$ is the Markov operator
of $\mathcal{F}p^k$, when $M$ is the Markov operator
of $\mathcal{F}p$. 
Further, by the average contractivity,
$M^k$ is asymptotically stable according to
\cite[Proposition 12.8.1]{LasotaMackey} 
(which can be also found as \cite[Fact 3.2]{MyjakSzarek}).
Thus $M^k$ admits a unique invariant measure $\mu_{*}$
supporting the semiattractor of $\mathcal{F}p^k$,
see \cite[Theorem 6.1]{LasotaMyjak2000}. 
Hence $M$ is asymptotically stable with $\mu_{*}$
as its invariant measure, due to the standard trick 
concerning contractive fixed points of higher iterates,
see \cite[Remark 1.2(3)]{LSS-Manual}. 
As before, $\mu_{*}$ supports the semiattractor
of $\mathcal{F}p$.
\end{proof}

\begin{remark}
Part (a) of Theorem \ref{th:criterion3} can be deduced
from the proof of Part (b). 
\end{remark}

Later on we will apply the above criterion 
when $X=\R^2$ and the maps comprising an IFS are affine. 
Therefore, let us recall how one can compute 
the Lipschitz constant $\Lip(f)$ 
for an affine map $f$ on $\R^d$ 
equipped with the $2$-norm, $\|\cdot\|_2$.
If $f(x)=A\, x+B$ for $x\in\R^d$, then
\begin{align*}
	\Lip(f) 
	=\sup_{x\neq y} 
	\frac{\|f(x)- f(y)\|_2}{\|x- y\|_2}
	=\|A\|_2 
	=\sigma_{\max}(A),
\end{align*}
where $\sigma_{\max}(A)$ is 
the maximum singular value of $A$, 
see, for example, \cite[Chap. 5.2]{Meyer}.

\section{Strongly-fibred IFS}

We introduce the classification of invariant sets of IFSs
due to B. Kieninger, see \cite{Kieninger, BV2013}. 

Let $A_*\subseteq X$ be a compact minimal invariant set 
(strict attractor, semiattractor) of an IFS $\mathcal{F}$. 
We define a \notion{coding map} 
$\pi: I^{\infty}\to 2^{A_{*}}$ by the formula 
\begin{align}\label{eq:fiber}
	\pi(i_{1}i_{2}...) = 
	\bigcap_{n=1}^{\infty} 
	f_{i_1}\circ\dots\circ f_{i_n}(A_{*})
	\;\mbox{ for }\; 
	(i_n)_{n=1}^{\infty}\in I^{\infty}.
\end{align}
A set $\pi(i_{1}i_{2}\dots)$ is called 
a \notion{fibre} of $A_{*}$ with the 
\notion{address} $i_{1}i_{2}\dots$; 
it is a nonempty compact subset of $A_{*}$.
In consequence, $A_{*}$ is a union of fibres,
\begin{align*}
	A_{*} = \pi(I^{\infty}) = 
	\bigcup_{(i_n)_{n=1}^{\infty}\in I^{\infty}}
	\pi(i_{1}i_{2}\dots).
\end{align*}

Let us remark that if $A_*$ is 
a strict attractor of $\mathcal{F}$ 
and we replace in \eqref{eq:fiber} $A_*$ 
with any compact set $C\supseteq A_*$ such that 
$F(C)\subseteq C\subseteq \mathcal{B}(A_{*})$, 
then the resulting set will coincide 
with the fibre as defined in \eqref{eq:fiber}, 
see \cite[Proposition 1]{BL}. 
Note that in \cite{Kieninger}, 
the author considers more general framework and 
he takes $C=X$ where $X$ is a compact Hausdorff space.

\begin{definition}
We say that a compact minimal invariant set $A_{*}$ 
of an IFS $\mathcal{F}$ is: 
\begin{itemize}
	\item[(i)] \notion{point-fibred}, 
	if all its fibres are singletons; 
	\item[(ii)] \notion{strongly-fibred}, 
	if for every $a\in A_{*}$ 
	and open $V\ni a$ there exists 
	$(i_n)_{n=1}^{\infty}\in I^{\infty}$
	such that $\pi(i_{1}i_{2}...)\subseteq V$.
\end{itemize}
\end{definition}  

From the above definition, it is clear 
that a minimal invariant set which is point-fibred, 
is also strongly-fibred. Further, let us note that 
if an IFS is weakly contractive, then its strict attractor
is point-fibred. And vice-versa, 
if $A_{*}$ is a point-fibred minimal invariant
set of the IFS $\mathcal{F}$, then $\mathcal{F}$
can be remetrized to an IFS of weakly contractive 
maps on $A_{*}$, see \cite{BKNNS, Miculescu}. 
For a systematic account of basic types of 
weakly contractive IFSs we refer the reader 
to \cite{LSS-Manual}.  
An IFS with a strongly-fibred but 
not point-fibred minimal invariant set
is not weakly contractive under any remetrization. 
For recent generalizations of strong-fibredness 
we refer the reader to 
\cite{BarrientosEtAl2017} and \cite{Diaz2020}.

\begin{theorem}[The Kieninger criterion]\label{th:criterionK}
Let $\mathcal{F} = (X; f_i : i\in I)$ be an IFS.
Let $A_{*}$ be either a compact semiattractor or 
a strict attractor of $\mathcal{F}$. 
Suppose that $A_{*}$ admits a singleton
fibre, i.e., there exists 
$(i_n)_{n=1}^{\infty}\in I^{\infty}$
such that $\pi(i_{1}i_{2}\dots)=\{a_{*}\}$
for some $a_{*}\in A_{*}$. 
Then $A_{*}$ is strongly-fibred. 
\end{theorem}
\begin{proof}
Fix $a\in A_{*}$ and open $V\ni a$. 
Recall that $F^m(a_{*}) \to A_{*}$ when $m\to\infty$
(using Theorem \ref{th:relationsbetweenattractors} (iv), 
when $A_*$ is a semiattractor).
Hence $F^m(a_{*})\cap V\neq\emptyset$ 
for sufficiently large $m\geq 1$.
Observe that 
\begin{align*}
	F^m(a_{*}) = 
	{\{f_{j_1}\circ\dots\circ f_{j_m}(a_{*}) : 
	j_1,\dots,j_m\in I\}}.
\end{align*}
Therefore there exist $j_1,\dots,j_m\in I$ such that 
$f_{j_1}\circ\dots\circ f_{j_m}(a_{*}) \in V$.
Finally, we arrive at
\begin{align}
	\pi(j_{1}\dots j_{m}i_{1}i_{2}\dots) 
	&=
	\bigcap_{n=1}^{\infty} 
	f_{j_1}\circ\dots\circ f_{j_m}\circ 
	f_{i_1}\circ\dots\circ f_{i_n}(A_{*})  
	\label{eq:intersection1}
	\\
	&=
	f_{j_1}\circ\dots\circ f_{j_m}
	\left(\bigcap_{n=1}^{\infty} 
	f_{i_1}\circ\dots\circ f_{i_n}(A_{*})\right)  
	\label{eq:intersection2}
	\\
	&=
	f_{j_1}\circ\dots\circ f_{j_m}(\{a_{*}\}) 
	\subseteq V. 
	\nonumber
\end{align}
To pass from \eqref{eq:intersection1} 
to \eqref{eq:intersection2}, 
we note that the sets are nested and compact.
The proof is complete as $V$ was 
an arbitrary open set with 
$V\cap A_{*}\neq \emptyset$.
\end{proof}

Note that the above criterion is valid 
for a general topological space $X$. 
It could be deduced directly from  
\cite[Proposition 4.3.16]{Kieninger}
in the case where the singleton fibre 
$\pi(i_{1}i_{2}\dots)$ has a constant address,
$i_{1}=i_{2}=\dots$. 
For a metric space $X$ the reverse criterion is true and is due  to E. Matias (see 
\cite[Proposition 4.1, Acknowledgements]{BarrientosEtAl2020}): 
If $A_{*}$ is an invariant set 
which is strongly-fibred, then 
it admits at least one singleton fibre.

\section{Strongly-fibred triangle}

In this section an IFS which was introduced 
in \cite{BV2013} will be analysed 
from the perspective of the Lasota--Myjak theory 
as discussed in Section \ref{sec:semiattractors}.
		
Let $\Delta\subseteq \R^2$ be the filled triangle 
with vertices $(0,0),(1,0)$ and $(0,1)$
on the Euclidean plane. 
Let $f_i:\R^2\to\R^2$, 
$f_i(x,y) = A_i 
\left(\begin{smallmatrix}
x \\ y
\end{smallmatrix}\right) 
+ \left(\begin{smallmatrix}
0 \\ i-1
\end{smallmatrix}\right)$ 
for $(x,y)\in\R^2$, $i=1,2$, where 
\begin{align*}
	A_1=\left(\begin{array}{rr}
	1&\frac{1}{2}\\
	0&\frac{1}{2}
	\end{array}\right), \quad 
	A_2=\left(\begin{array}{rr}
	0&\frac{1}{2}\\
	-1&-\frac{1}{2}
	\end{array}\right).
\end{align*}	
Clearly, $\Delta$ is an invariant set for $\mathcal{F}$, 
i.e., $\Delta=f_1(\Delta)\cup f_2(\Delta)$, 
see Figure \ref{fig:triangle}.
However, as it has been stated in \cite{BV2013}, 
the IFS $\mathcal{F} = (\R^2; f_1, f_2)$ 
does not admit a strict attractor. On the other hand, 
its restriction $\mathcal{F}_{\Delta} = (\Delta; f_1,f_2)$ 
has $\Delta$ as a strict attractor which is strongly-fibred.


\begin{figure}[h]
\begin{center}
\begin{tikzpicture}[line cap=round,line join=round,>=triangle 45,x=3cm,y=3cm]
\clip(-0.5,-0.2) rectangle (1.3,1.2);
\draw [line width=2.pt] (0.,0.)-- (1.,0.);
\draw [line width=2.pt] (1.,0.)-- (0.,1.);
\draw [line width=2.pt] (0.,1.)-- (0.,0.);
\draw [line width=2.pt] (0.5,0.5)-- (0.,0.);
\draw (-0.2,-0.05) node[anchor=north west] {$(0,0)$};
\draw (0.9,-0.05) node[anchor=north west] {$(1,0)$};
\draw (-0.2,1.2) node[anchor=north west] {$(0,1)$};
\draw (0.5,0.7) node[anchor=north west] {$\left(\frac{1}{2}, \frac{1}{2}\right)$};
\draw (0.3,0.3) node[anchor=north west] {$f_1(\Delta)$};
\draw (0.05,0.6) node[anchor=north west] {$f_2(\Delta)$};
\end{tikzpicture}
\end{center}
\caption{The action of $f_1$ and $f_2$ on 
the strongly-fibred triangle $\Delta$.
$f_1$ maps the vertices of $\Delta$ as follows:
$(0,0)\mapsto (0,0)$,
$(1,0)\mapsto (1,0)$,
$(0,1)\mapsto (1/2,1/2)$.
$f_2$ maps the vertices of $\Delta$ as follows:
$(0,0)\mapsto (0,1)$,
$(1,0)\mapsto (0,0)$,
$(0,1)\mapsto (1/2,1/2)$.
}
\label{fig:triangle}
\end{figure}
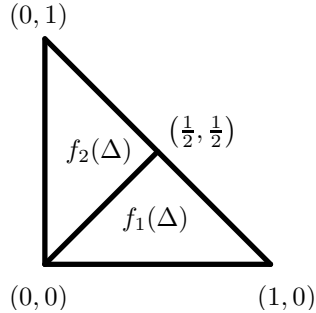

We will show much more. Namely, $\Delta$ is a strongly-fibred 
semiattractor of $\mathcal{F}$. 
(A fortiori $\Delta$ is the strict attractor 
of $\mathcal{F}_{\Delta}$, thanks to 
Theorem \ref{th:relationsbetweenattractors} (iv).)
Moreover, the normalized two-dimensional Lebesgue measure 
$\mathcal{L}$ restricted to $\Delta$, 
\begin{align*}
	\mu(B) = 
	\frac{\mathcal{L}(B\cap\Delta)}{\mathcal{L}(\Delta)} 
	= 2\mathcal{L}(B\cap \Delta)
	\mbox{ for Borel } B\subseteq \R^2, 
\end{align*}
is the unique invariant measure of the probabilistic IFS 
$\mathcal{F}p = (\R^2; (f_i, p_i) : i\in\{1,2\})$
with $p_1=p_2=\frac{1}{2}$.
	
First we show that $\mu$ is an invariant probability 
measure of $\mathcal{F}p$. Let $B$ be a Borel 
subset of $\R^2$. Let us write $B$ 
as $B=B_1\cup B_2\cup B_3$ where 
$B_1 \subseteq f_1(\Delta)$, 
$B_2 \subseteq f_2(\Delta)$ and 
$B_3 \subseteq \R^2\setminus \Delta$. 
Then we have
\begin{align*}
	f_i^{-1}(B_j)\subseteq 
	\R^2\setminus \inte{\Delta} 
	\quad  \text{for}\; 
	i\in \{1,2\}, j \in \{1,2,3\} 
	\;\text{and}\; i\neq j,
\end{align*}
where $\inte{}\!$ denotes the interior of a set. 

Hence for $i\in \{1,2\}$, 
$j \in \{1,2,3\}$ and $i\neq j$ we have
\begin{align*}
	0\leq \mu(f_i^{-1}(B_j))=
	2\mathcal{L} (f_i^{-1}(B_j)\cap\Delta)\leq 
	2\mathcal{L} ((\R^2\setminus \inte{\Delta})\cap\Delta)=0.
\end{align*}
We will now show that for $i=1,2$, 
\begin{align*}
	\mu (f_i^{-1}(B_i))=2\mu(B_i). 
\end{align*}
Since both maps $f_1$ and $f_2$ are affine linear 
with $\det A_1 =\det A_2= \frac{1}{2}$, we have
\begin{align*}
	\mu(f_i^{-1}(B_i))=
	2\mathcal{L} (f_i^{-1}(B_i)\cap \Delta)=
	2\mathcal{L} (f_i^{-1}(B_i))=
	4\mathcal{L}(B_i)=
	4\mathcal{L}(B_i\cap\Delta)=
	2\mu(B_i).
\end{align*}

We also note that $\mu(f_1(\Delta)\cap f_2(\Delta))=0$. 
This implies that
$\mu(f_j^{-1}f_i(\Delta))=0$ for $i\neq j$.

Taking into account the above observations, we have
\begin{align*}
	M \mu (B)
	& 
	=
	\frac{1}{2}\mu(f_1^{-1}(B))+\frac{1}{2}\mu(f_2^{-1}(B))
	\\
	\mbox{}
	&	
	= \frac{1}{2}\mu(f_1^{-1}(B_1)\cup f_1^{-1}(B_2)\cup f_1^{-1}(B_3))
	+ \frac{1}{2}\mu(f_2^{-1}(B_1)\cup f_2^{-1}(B_2)\cup f_2^{-1}(B_3))
	\\
	\mbox{}
	&
	= \frac{1}{2}\mu(f_1^{-1}(B_1))+\frac{1}{2}\mu( f_2^{-1}(B_2))
	\\
	\mbox{}
	&
	= \mu (B_1)+\mu(B_2)=\mu(B).
\end{align*}

So far we know that $\Delta$ is 
an invariant set of $\mathcal{F}$ and 
$\mu$ is an invariant probability measure of $\mathcal{F}p$. 
To prove that $\Delta$ is, indeed, 
a semiattractor of $\mathcal{F}$ 
and $\mu$ is unique, let us show that
$\mathcal{F}^{2}p=(\R^2; (f_i\circ f_j, 1/4) : i,j\in\{1,2\})$ 
is contractive on average.

In order to do this, let us compute 
the Lipschitz constants of the maps 
which comprise $\mathcal{F}^{2}p$. 
Using the maximum singular values of $A_i A_j$, we get
\begin{align*}
	\Lip(f_i\circ f_2)= \|A_i A_2\|_2= 
	\frac{1}{\sqrt{2}}, i=1,2, 
	\\
	\Lip(f_i\circ f_1)= \|A_i A_1\|_2= 
	\frac{\sqrt{3\sqrt{17}+13}\,}{4}, i=1,2.
\end{align*}
Hence we have
\begin{align*}
	\sum_{i,j=1}^{2} \Lip(f_i\circ f_j) < 4
\end{align*}
and so $\mathcal{F}^{2}p$ is contractive on average. Using the Lasota-Myjak criterion 
(Theorem \ref{th:criterion3}), we conclude that 
$\mu$ is a unique invariant probability measure
for $\mathcal{F}p$ and $\Delta =\supp\mu$ 
is a semiattractor of $\mathcal{F}$.

Let us note that for the average contractivity of 
$\mathcal{F}^{2}p$, it is enough to assume that
$p_1<\frac{4-2\sqrt{2}}{\sqrt{3\sqrt{17}+13}-2\sqrt{2}}
\approx 0.53$. So according to Theorem \ref{th:criterion3} and the fact that $\Delta$ is a semiattractor of $\mathcal{F}$, there is a unique invariant probability measure for $\mathcal{F}p$, where $p_1<0.53$, supported on $\Delta$. This gives rise to a natural question if such invariant measure is singular or absolutely continuous with respect to Lebesgue measure. One can also ask, given $p_1\in [0.53,1]$, does there exist a natural number $k$ such that $\mathcal{F}^{k}p$ is contractive on average?

Finally, we address the question 
how the semiattractor $\Delta$
is fibred by the IFS $\mathcal{F}$.
Let us look at two particular fibres, 
$\pi(i\overline{i})$, $i=1,2$,
where $\overline{i}$ is a constant 
infinite sequence of $i$'s.
By simple calculation
\begin{align*}
	A_1^n
	= 
	\left(\begin{array}{rr}
	1&\sum_{k=1}^{n}\frac{1}{2^k}\\
	0&\frac{1}{2^n}
	\end{array}\right)
	= 
	\left(\begin{array}{rr}
	1&1-2^{-n}\\
	0&2^{-n}
	\end{array}\right)
\end{align*}
we get that $f_1^{n}(\Delta)$ 
is the triangle with vertices 
$(0,0)$, $(1,0)$ and $(1-2^{-n},2^{-n})$,
from which it follows that
$\pi(1\overline{1}) = [0,1]\times \{0\}$.
Thus $\Delta$ is not point-fibred.
Next, since $f_2^2$ is a contraction, we get 
$\pi(2\overline{2}) = \{(\frac{1}{4},\frac{1}{2})\}$.
According to the Kieninger criterion 
(Theorem \ref{th:criterionK}) $\Delta$ 
is strongly-fibred by $\mathcal{F}$.
To complete the picture of fibre structure of $\Delta$,
let us note that all fibres are 
either points or intervals, 
see \cite[Figure 7]{BV2013}.
To see this, observe that each fibre
\begin{align*}
	\pi(i_{1}i_{2}\dots) = 
	\bigcap_{n=1}^{\infty} 
	f_{i_1}\circ\dots\circ f_{i_n}(\Delta)
\end{align*}
is an intersection of nested triangles
whose area is shrinked by factor $1/2$ 
downward the nest, i.e.,
\begin{align*}
	\mathcal{L}(f_{i_1}\circ\dots\circ f_{i_{n+1}}(\Delta))
	= \frac{1}{2}\cdot 
	\mathcal{L}(f_{i_1}\circ\dots\circ f_{i_{n}}(\Delta))
\end{align*}
Therefore, $\pi(i_{1}i_{2}\dots)$ is a nonempty,
compact and convex subset of the plane 
whose area is $0$, that is, 
either a point, or an interval.

The analysis of the above example 
was not straightforward because the IFS at hand 
was neither contractive on average
(\cite[Theorem 2.60]{Mendivil}) 
nor eventually contractive 
(\cite[Example 2.28]{Mendivil}). 
Fortunately, the Markov operator turned out 
to be eventually average-contractive.
The mere existence of the semiattactor 
could have been easily inferred from 
the contractivity of $f_2^2$ 
due to \cite[Theorem 3.2]{LasotaMyjak2000} 
(which is \cite[Theorem 6.3]{MyjakSzarek}). 
However, we would have not known 
if $\Delta$ is that semiattractor.

Techniques, similar to the ones 
we used in the above example, 
can be applied to other non-contractive affine IFSs. 
For instance, we have the following example.

\begin{example}
Let $\mathcal{F}p = (\R^2; (f_i, 1/2) : i\in\{1,2\})$,
where $f_1(x,y)=(x/2+1/2,y)$, $f_2(x,y)=(x,y/2)$ 
for $(x,y)\in\R^2$. We proceed as before. Noting that
\begin{align*}
	\Lip(f_i\circ f_i)=1,\; 
	\Lip(f_{3-i}\circ f_i)=1/2,\;
	i=1,2,
\end{align*}
we get that the Markov operator $M$ induced by 
$\mathcal{F}p$ is asymptotically stable.
Thus $\mathcal{F}p$ admits 
a unique invariant probability measure $\mu_{*}$
and a semiattractor $A_{*}$. 
It is easy to see that $A_{*}= \{(1,0)\}$, 
because the common fixed point of $f_1$ 
and $f_2$ is a minimal invariant set, 
and a fortiori $\mu_{*}= \delta_{(1,0)}$, 
the Dirac measure supported at the point $(1,0)$.
\end{example}



\begin{thebibliography}{HD}
 







\bibitem{BKNNS} T. Banakh, W. Kubi\'s, M. Nowak, N. Novosad, F. Strobin, 
	\emph{Contractive function systems, their attractors and metrization}, 
	Topol. Methods Nonlinear Anal. {46} (2015), no.~2, 1029--1066.

\bibitem{BarnsleyElton} M.F. Barnsley, J.H. Elton,
	\emph{A new class of Markov processes for image encoding},
	Adv. Appl. Prob. {20} (1988), 14--32.

\bibitem{BL} M. F. Barnsley, K. Leśniak, 
	\emph{The chaos game on a general iterated function system 
	from a topological point of view}, 
	Int. J. Bifurcation Chaos Appl. Sci. Eng. {24} (2014), 
	no.~11, Article ID 1450139, 10 p.  

\bibitem{BLR2016}  M. Barnsley, K. Le\'{s}niak, M. Rypka, 
	\emph{Chaos game for IFSs on topological spaces},
	J. Math. Anal. Appl. {435} (2016), no.~2, 1458--1466.

\bibitem{BV2013} M. Barnsley, A. Vince, 
	\emph{Developments in fractal geometry},
	Bull. Math. Sci. {3} (2013), no.~2, 299--348.

\bibitem{BarrientosEtAl2020} P.G. Barrientos, M. Fitzsimmons, 
	F.H. Ghane, D. Malicet, A. Sarizadeh,
	\emph{Addendum and corrigendum to 
	``On the chaos game of iterated function systems''},
	Topol. Methods Nonlinear Anal. {55} (2020), no.~2, 601--616.

\bibitem{BarrientosEtAl2017} P.G. Barrientos, F.H. Ghane, 
	D. Malicet, A. Sarizadeh, 
	\emph{On the chaos game of iterated function systems},
	Topol. Methods Nonlinear Anal. {49} (2017), no.~1, 105--132. 

\bibitem{Beer} G. Beer, 
	\emph{Topologies on Closed and Closed Convex Sets},
	Kluwer, Dordrecht, 1993. 
	
\bibitem{Billingsley} P. Billingsley, 
	\emph{Convergence of Probability Measures}, 
	2nd ed., Wiley, New York, 1999.
	

\bibitem{Ed} A. Edalat, \emph{Power domains and iterated function systems}, Inf. Comput. 124 (1996), no. 2, 182--197.	

	
\bibitem{Diaz2020} L.J. D{\`\i}az, E. Matias,
	\emph{Non-hyperbolic iterated function systems: 
	semifractals and the chaos game},
	Fund. Math. {250} (2020), no.~1, 21--39. 
	

	\bibitem{Kam} A. Kameyama,\emph{On the self-similar sets with frames}, The study of dynamical systems. Papers from the conference, held in Kyoto, Japan, February 1-4, 1989, 1--9, Singapore: World Scientific Publishing, 1989. 


\bibitem{Kieninger} B. Kieninger, 
	\emph{Iterated Function Systems on Compact Hausdorff Spaces}, 
	Shaker-Verlag, Aachen, 2002.
	
\bibitem{Mendivil} H. Kunze, D. La Torre, F. Mendivil, E. Vrscay, 
	\emph{Fractal-based Methods in Analysis}, 
	Springer, Berlin, 2012.

\bibitem{LasotaMackey} A. Lasota, M.C. Mackey, 
	\emph{Chaos, Fractals, and Noise. 
	Stochastic Aspects of Dynamics}, 
	2nd ed., Springer, New York, 2004.
	
\bibitem{LasotaMyjak} A. Lasota, J. Myjak, 
	\emph{Semifractals}, 
	Bull. Pol. Acad. Sci. Math. {44} (1996), no.~1, 5--21. 

\bibitem{LasotaMyjak2000} A. Lasota, J. Myjak, 
	\emph{Attractors of multifunctions}, 
	Bull. Pol. Acad. Sci. Math. {48} (2000), no.~3, 319--334.

\bibitem{LSS-Manual} K. Le\'{s}niak, N. Snigireva, F. Strobin,
	\emph{Weakly contractive iterated function systems and beyond: 
	a manual}, 
	J. Difference Equ. Appl. {26} (2020), no.~8, 1114--1173.
	
\bibitem{Meyer} C. Meyer, 
	\emph{Matrix Analysis and Applied Linear Algebra}, 
	SIAM, Philadelphia, PA, 2000.

\bibitem{Miculescu} R. Miculescu, A. Mihail,
	\emph{On a question of A. Kameyama concerning self-similar metrics},
	J. Math. Anal. Appl. {422} (2015), no.~1, 265--271.
	 
\bibitem{MyjakSzarek} J. Myjak, T. Szarek,
	\emph{Attractors of iterated function systems and Markov operators},
	Abstr. Appl. Anal. {2003} (2003), no.~8, 479--502. 

\bibitem{Zaharopol}  R. Zaharopol, 
	\emph{Invariant Probabilities of Markov-Feller Operators and Their Supports},
	Birkh\"{a}user, Basel, 2005.

\end{thebibliography}
\end{document}